\documentclass[12pt]{article}
\usepackage{amssymb}
\usepackage{amsfonts}
\usepackage{amsthm}
\usepackage{amsmath}
\usepackage{fullpage}
\usepackage{algorithm}
\usepackage{algorithmic}
\usepackage{mathtools}

\newtheorem{Theorem}{Theorem}

\newtheorem{Corollary}{Corollary}

\newcommand{\vol}{\text{vol}}

\newcommand{\conv}{\text{conv}}
\newcommand{\area}{\text{area}}

\usepackage[pdftex]{graphicx}
\graphicspath{{C:/Users/Sam/Desktop/PaperFigures/}}
\DeclareGraphicsExtensions{.pdf,.png}

\begin{document}
\title{Upper Bounds for Non-Congruent Sphere Packings}
\author{Samuel Reid\thanks{University of Calgary, Centre for Computational and Discrete Geometry (Department of Mathematics \& Statistics), Calgary, AB, Canada. $\mathsf{e-mail: smrei@ucalgary.ca}$}}
\maketitle

\begin{abstract}
We prove upper bounds on the average kissing number $k(\mathcal{P})$ and contact number $C(\mathcal{P})$
of an arbitrary finite non-congruent sphere packing $\mathcal{P}$, and prove an upper bound on the packing density
$\delta(\mathcal{P})$ of an arbitrary infinite non-congruent sphere packing $\mathcal{P}$.
\end{abstract}

\textbf{Keywords:} average kissing number, contact number, packing density, non-congruent sphere packing, linear programming bounds, upper bounds. \\
\text{   \;\;       } \textbf{MSC 2010 Subject Classifications:} Primary 52C17, Secondary 51F99.

\section{Lexicon}
Let
$$\mathcal{P} = \bigcup_{i=1}^{k} \bigcup_{j=1}^{n_{i}} \left(x_{ij} + r_{i}\mathbb{S}^{2}\right)$$
be an arbitrary non-congruent sphere packing. Then 
$$\|x_{ij} - x_{i'j'}\| \geq r_{i} + r_{i'}, \forall 1 \leq i,i' \leq k, j \neq j'$$
is a necessary condition required for the spheres to be non-overlapping. Hence, the vertex set and edge set of the
sphere packing $\mathcal{P}$ are
\begin{align*}
 V(\mathcal{P}) &= \{x_{ij} \in \mathbb{R}^{3} \; \big| \; 1 \leq i \leq k, 1 \leq j \leq n_{i}\} \\
 E(\mathcal{P}) &= \{(x_{ij},x_{i',j'}) \; \big| \; \|x_{ij}-x_{i'j'}\| = r_{i} + r_{i'}, x_{ij},x_{i',j'} \in V(\mathcal{P})\}
\end{align*}

The the average kissing number of $\mathcal{P}$ is $k(\mathcal{P}) = 2|E(\mathcal{P})|/n$,
where $n = \displaystyle \sum_{i=1}^{k} n_{i}$, and the contact number of $\mathcal{P}$ is $C(\mathcal{P})=|E(\mathcal{P})|$.

\section{Upper Bounds on Average Kissing Numbers and Contact Numbers of Sphere Packings}
Consider, as in Cohn-Zhao \cite{CohnZhao}, a continuous function $g:[-1,1] \rightarrow \mathbb{R}$ which is positive definite on $\mathbb{S}^2$ with 
$g(t) \leq 0, \forall t \in [-1,\cos\theta]$ and
$$\overline{g} = \frac{\int_{-1}^{1} g(t)(1-t^2)dt}{\int_{-1}^{1} (1-t^{2})dt}>0,$$
then $$g \in \mathcal{F}_{\theta}(\mathbb{S}^2).$$
Furthermore, let $A(3,\theta)$ be the maximum size of a spherical $\theta$-code and recall that
$$A^{\text{LP}}(3,\theta) = \inf_{g \in \mathcal{F}_{\theta}(\mathbb{S}^{2})} \frac{g(1)}{\overline{g}}$$
is the best upper bound on $A(3,\theta)$ that could be derived using Theorem 3.1 from Cohn-Zhao \cite{CohnZhao} (which appeals
to the Delsarte-Goethals-Seidel \cite{DGS} and Kabatiansky-Levenshtein \cite{KL} linear programming bounds). 
Let $\tau_{r_{j}}(r_{i})$ be the maximum number of radius $r_{i}$ spheres which can touch a radius $r_{j}$ sphere; 
$\tau_{r_{i}}(r_{j})$ is defined similarly,
namely the maximum number of radius $r_{j}$ spheres which can touch a radius $r_{i}$ sphere.

\begin{Theorem}\label{contacttheorem}
Let $\mathcal{P}$ be a sphere packing with $n_{i}$ spheres of radius $r_{i}$ for $1 \leq i \leq k$. Then,
$$k(\mathcal{P}) < 12 + \frac{\displaystyle \sum_{i \neq j} \min\left\{n_{i}\min\left\{n_{j},\tau_{r_{i}}(r_{j})\right\},n_{j} \min\left\{n_{i},\tau_{r_{j}}(r_{i})\right\}\right\} - 1.85335 \sum_{i=1}^{k}n_{i}^{2/3}}{\displaystyle \sum_{i=1}^{k}n_{i}}$$
where $\tau_{r_{i}}(r_{j}) \leq A^{\text{LP}}\left(3,\arccos\left(1 - \frac{2r_{j}^{2}}{(r_{i}+r_{j})^{2}}\right)\right)$ and
$\tau_{r_{j}}(r_{i}) \leq A^{\text{LP}}\left(3,\arccos\left(1 - \frac{2r_{i}^{2}}{(r_{i}+r_{j})^{2}}\right)\right)$.
\end{Theorem}
\begin{proof}
Decompose the vertex set and edge set of $\mathcal{P}$ as
\begin{align*}
V(\mathcal{P}) &= \bigcup_{i=1}^{k} V_{i}(\mathcal{P}) := \bigcup_{i=1}^{k} \{ x_{ij} \in V(\mathcal{P}) \; | \; 1 \leq j \leq n_{i} \} \\
|V(\mathcal{P})| &= \sum_{i=1}^{k}|V_{i}(\mathcal{P})| = \sum_{i=1}^{k}n_{i} \\
E(\mathcal{P}) &= \bigcup_{i = 1}^{k} E_{ii}(\mathcal{P}) \cup \bigcup_{i \neq i'} E_{ii'}(\mathcal{P}) :=\bigcup_{i=1}^{k}\{(x_{ij},x_{ij'}) \; | \|x_{ij} - x_{ij'}\| = 2r_{i}, x_{ij},x_{ij'} \in V_{i}(\mathcal{P})\} \\
&\cup \bigcup_{i \neq i'} \{(x_{ij},x_{i'j'}) \; \big| \; \|x_{ij} - x_{i'j'} \| = r_{i} + r_{i'}, x_{ij} \in V_{i}(\mathcal{P}),x_{i'j'} \in V_{i'}(\mathcal{P})\} \\
|E(\mathcal{P})| &= \sum_{i=1}^{k}|E_{ii}(\mathcal{P})| + \frac{1}{2}\sum_{i \neq j}|E_{ij}(\mathcal{P})|
\end{align*}

Apply Theorem 1 (i) of Bezdek and the author \cite{BezdekReid} (Theorem 1.1.6 (i) in \cite{Bezdek}) to bound the cardinality
of each edge set and obtain $|E_{ii}|<6n_{i} - 0.926n_{i}^{2/3}, \forall 1 \leq i \leq k$. 
By applying a homothetic transformation with a scaling factor of either $1/r_{i}$ or $1/r_{j}$ to $\mathcal{P}$, 
it is clear that $\tau_{r_{j}}(r_{i}) = \tau_{1}(r_{i}/r_{j})$ and $\tau_{r_{i}}(r_{j}) = \tau_{1}(r_{j}/r_{i})$.
Hence, each $|E_{ij}(\mathcal{P})| + |E_{ji}(\mathcal{P})|$
counts the number of edges between spheres of radius $r_{i}$ and $r_{j}$, and thus by the above homothetic transformation
of either type, counts the number of edges between spheres of radius $1$ and radius $r_{j}/r_{i}$ or spheres
of radius $r_{i}/r_{j}$ and radius $1$, respectively.
Therefore, by the law of cosines applied to the geometric embedding of 
the contact graph of two spheres of radius $r_{j}/r_{i}$ and a sphere of radius $1$, and the 
geometric embedding of the contact graph of two spheres of radius $r_{i}/r_{j}$ and a sphere of radius $1$,
we obtain the $\theta$-code size desired in each case:
\begin{align*}
\theta_{i}^{j} &= \arccos\left(1 - \frac{2r_{j}^{2}}{(r_{i} + r_{j})^{2}}\right) \\
\theta_{j}^{i} &= \arccos\left(1 - \frac{2r_{i}^{2}}{(r_{i} + r_{j})^{2}}\right)
\end{align*}
Hence, $\tau_{1}(r_{i}/r_{j}) = A(3,\theta_{j}^{i}) \leq A^{\text{LP}}(3,\theta_{j}^{i})$ and 
$\tau_{1}(r_{j}/r_{i})=A(3,\theta_{i}^{j}) \leq A^{\text{LP}}(3,\theta_{i}^{j})$.
From this we observe, from basic restrictions on the number of spheres of varying radii, that
\begin{align*}
|E_{ij}(\mathcal{P})|=|E_{ji}(\mathcal{P}| &< \min\left\{n_{i}\min\left\{n_{j},\tau_{r_{i}}(r_{j})\right\},n_{j} \min\left\{n_{i},\tau_{r_{j}}(r_{i})\right\}\right\} \\
&=\min\left\{n_{i}\min\left\{n_{j},A^{\text{LP}}(3,\theta_{i}^{j})\right\},n_{j} \min\left\{n_{i},A^{\text{LP}}(3,\theta_{j}^{i})\right\}\right\}
\end{align*}
Cumulatively, these observations combined with the definition of the average kissing number $k(\mathcal{P})$ prove the theorem.
\end{proof}

\pagebreak

In practice, it is difficult to compute $A^{\text{LP}}(3,\theta)$ explicitly, but a weaker bound may be provided by nonnegative
linear combinations of Gegenbauer polynomials $C_{k}^{\frac{n}{2}-1}$ as shown by Schoenbergs characterization
of continuous positive definite functions \cite{Schoenberg}. Gegenbauer polynomials, or ultraspherical polynomials, are
a special case of the Jacobi polynomials, or hypergeometric polynomials, and for algorithmic implementations of the
following theorem we can follow \cite{CohnZhao} and set
$$g(t) = \sum_{k=0}^{\infty} c_{k}C_{k}^{\frac{n}{2}-1}(t), \overline{g}=c_{0}$$
For algorithmic implementation, tighter bounds on $A(n,\theta)$ may be found using de Laat-de Oliveira Filho-Vallentin
semidefinite programming bounds \cite{LOV}, or Cohn-Elkies error-correcting codes bounds \cite{CohnElkies}.
Furthermore, Theorem \ref{contacttheorem} can be considered a packing dependent generalization of the
celebrated Kuperberg-Schramm bound on the supremal average 
kissing number $k$ of a sphere packing in $\mathbb{R}^3$ \cite{KuperbergSchramm}, which says that 
$12.566 < k := \sup_{\mathcal{P} \hookrightarrow \mathbb{R}^3} k(P) < 8+4\sqrt{3} \approx 14.928$. Future
research goals include the algorithmic implementation of Theorem \ref{contacttheorem} to compare with the Kuperberg-Schramm bound.

We now state Theorem \ref{contacttheorem} in terms of contact numbers which follows directly from the definition
of $k(\mathcal{P})$, thus generalizing Theorem 1 (i) of Bezdek and the author \cite{BezdekReid} (Theorem 1.1.6 (i) in \cite{Bezdek}),
which states that if $\mathcal{P}$ is a packing of $n$ congruent spheres then $C(\mathcal{P}) < 6n - 9.26n^{2/3}$.

\begin{Corollary}
Let $\mathcal{P}$ be a sphere packing with $n_{i}$ spheres of radius $r_{i}$ for $1 \leq i \leq k$. Then,
$$C(\mathcal{P}) < \sum_{i=1}^{k} (6n_{i} - 0.926675n_{i}^{2/3}) + \frac{1}{2}\sum_{i \neq j} \min\left\{n_{i}\min\left\{n_{j},\tau_{r_{i}}(r_{j})\right\},n_{j} \min\left\{n_{i},\tau_{r_{j}}(r_{i})\right\}\right\}$$
where $\tau_{r_{i}}(r_{j}) \leq A^{\text{LP}}\left(3,\arccos\left(1 - \frac{2r_{j}^{2}}{(r_{i}+r_{j})^{2}}\right)\right)$ and
$\tau_{r_{j}}(r_{i}) \leq A^{\text{LP}}\left(3,\arccos\left(1 - \frac{2r_{i}^{2}}{(r_{i}+r_{j})^{2}}\right)\right)$.
\end{Corollary}

\section{Upper Bounds on Infinite Sphere Packings' Densities}
We define the locally maximal tetrahedron $\Delta(r_{i},r_{j},r_{k},r_{l})$ to be the convex hull
of the center points of spheres of radius $r_{i},r_{j},r_{k},r_{l}$, which are maximally contracted;
i.e., there does not exist a non-trivial contractive mapping of the spheres. We can use the geometric
structure of the locally maximal tetrahedron to calculate an upper bound on the density of an infinite sphere
packing of distinct radii $r_{i}, i \in S \subseteq \mathbb{N}$, by defining 
$\Delta(r_{i},r_{j},r_{k},r_{l}) = \conv\{\vec{\omega_{i}},\vec{\omega_{j}},\vec{\omega_{k}},\vec{\omega_{l}}\}$
and intersecting spheres of the associated radii at each vertex of the locally maximal tetrahedron.
By connecting spherical geometry and dihedral angles we arrive at the following theorem which holds
for any sphere packing $\mathcal{P}$ in $\mathbb{R}^3$ whether or not it has finitely
many distinct radii, or infinitely many distinct radii, although the theorem does not have a realizable
algorithmic implementation in the case of infinitely many distinct radii.

\begin{Theorem}
Let $\mathcal{P}$ be a sphere packing in $\mathbb{R}^3$ with distinct radii $r_{i}, i \in S \subseteq \in\mathbb{N}$, and let $\delta_{\text{max}}(\mathcal{P}_{\Delta(r_{i},r_{j},r_{k},r_{l})})$
be the maximal packing density of $\Delta(r_{i},r_{j},r_{k},r_{l})$ in $\mathbb{R}^3$. Then,
$$\delta(\mathcal{P}) < 2 \max_{r_{i}\leq r_{j} \leq r_{k} \leq r_{l}} \delta_{\text{max}}(\mathcal{P}_{\Delta(r_{i},r_{j},r_{k},r_{l})}) \left(\left[\sum_{m=i,j,k,l} r_{m}^{3}(A_{m} + B_{m} + C_{m} - \pi)\right] \bigg/ \|\vec{\omega_{2}} \cdot (\vec{\omega_{3}} \times \vec{\omega_{4}})\|\right),$$
where $\Delta(r_{i},r_{j},r_{k},r_{l}) = \conv\{\vec{\omega_{i}},\vec{\omega_{j}},\vec{\omega_{k}},\vec{\omega_{l}}\}$ and
\begin{align*}
 U_{ijk} &= \vec{\omega_{j}} \times \vec{\omega_{k}} \\
 U_{ikl} &= \vec{\omega_{k}} \times \vec{\omega_{l}} \\
 U_{ijl} &= \vec{\omega_{j}} \times \vec{\omega_{l}} \\
 U_{jkl} &= (\vec{\omega_{k}} - \vec{\omega_{j}}) \times (\vec{\omega_{l}} - \vec{\omega_{j}})
\end{align*}
\begin{align*}
A_{i}&=A_{l} = \arccos\left(\frac{U_{ijk} \cdot U_{ikl}}{\|U_{ijk}\| \|U_{ikl}\|}\right) \;\;\;\;\;\;\;\;\;\; A_{j}=A_{k} = \arccos\left(\frac{U_{ijk} \cdot U_{jkl}}{\|U_{ijk}\| \|U_{jkl}\|}\right) \\
B_{i}&=B_{j} = \arccos\left(\frac{U_{ijk} \cdot U_{ijl}}{\|U_{ijk}\| \|U_{ijl}\|}\right) \;\;\;\;\;\;\;\;\;\; B_{k}=B_{l} = \arccos\left(\frac{U_{ikl} \cdot U_{jkl}}{\|U_{ikl}\| \|U_{jkl}\|}\right) \\
C_{i}&=C_{k} = \arccos\left(\frac{U_{ikl} \cdot U_{ijl}}{\|U_{ikl}\| \|U_{ijl}\|}\right) \;\;\;\;\;\;\;\;\;\; C_{j}=C_{l} = \arccos\left(\frac{U_{ijl} \cdot U_{jkl}}{\|U_{ijl}\| \|U_{jkl}\|}\right).
\end{align*}
\end{Theorem}
\begin{proof}

Observe that
$$\vol((\vec{\omega_{m}}+r_{m}\mathbb{S}^{2}) \cap \Delta(r_{i},r_{j},r_{k},r_{l})) = \frac{r_{m}}{3} \area(\partial(\vec{\omega_{m}} + r_{m}\mathbb{S}^2) \cap \Delta(r_{i},r_{j},r_{k},r_{l})).$$
is the volume of a spherical wedge intersecting a sphere of radius $r_{m}$ and a locally maximal tetrahedron.
By calculating the volume of each of these spherical wedges in terms of the spherical area of a triangle on $r_{m}\mathbb{S}^2$
and observing that the supremum of $\delta(\mathcal{P})$ is less than the supremal density of a locally maximal tetrahedron $\Delta(r_{i},r_{j},r_{k},r_{l})$, we obtain
$$\delta(\mathcal{P}) < \max_{r_{i} \leq r_{j} \leq r_{k} \leq r_{l}} \frac{\delta_{\text{max}}(\mathcal{P}_{\Delta(r_{i},r_{j},r_{k},r_{l})}) \displaystyle \sum_{m=i,j,k,l} \vol\left((\vec{\omega_{m}} + r_{m} \mathbb{S}^2) \cap \Delta(r_{i},r_{j},r_{k},r_{l})\right)}{\vol(\Delta(r_{i},r_{j},r_{k},r_{l}))}$$
The apparatus for calculating the upper bound is self evident from the definition of dihedral angles,
tetrahedral volumes, and areas of spherical triangles.
\end{proof}

\pagebreak


\begin{thebibliography}{00}
\bibitem{BezdekReid}
K.~Bezdek, S.~Reid.
\newblock {Contact graphs of unit sphere packings revisited.}
\newblock Journal of Geometry, April 2013.

\bibitem{Bezdek}
K.~Bezdek.
\newblock {Lectures on Sphere Arrangements - the Discrete Geometric Side.}
\newblock Springer, 2013.

\bibitem{CohnZhao}
H.~Cohn, Y.~Zhao.
\newblock {Sphere Packing Bounds via Spherical Codes.}
\newblock Duke Math. J., Volume 163, Number 10, 2004; pp. 1965 - 2002.

\bibitem{DGS}
P.~Delsarte, J. M.~Goethals, J. J.~Seidel.
\newblock {Spherical codes and designs.}
\newblock Geometriae Dedicata, Volume 6, 1977; pp. 363 - 388.

\bibitem{KL}
G. A.~Kabatiansky, V. I.~Levenshtein.
\newblock {Bounds for packings on a sphere and in space. [Russian]}
\newblock Problemy Pereda$\breve{\text{c}}$i Informacii, Volume 14, 1978; pp. 3 - 25.

\bibitem{KuperbergSchramm}
G.~Kuperberg, O.~Schramm.
\newblock {Average kissing numbers for non-congruent sphere packings.}
\newblock Math. Res. Lett., Volume 1, 1994; pp. 339 - 344

\bibitem{Schoenberg}
I. J.~Schoenberg.
\newblock {Positive definite functions on spheres.}
\newblock Duke Math. J., Volume 9, 1942; pp. 96 - 108.

\bibitem{CohnElkies}
H.~Cohn, N.~Elkies.
\newblock {New upper bounds on sphere packings I.}
\newblock Annals of Mathematics, Volume 157, 2003; pp. 689 - 714.

\bibitem{LOV}
D.~de Laat, F. M.~de Oliveira Filho, F.~Vallentin.
\newblock {Upper bounds for packings of spheres of several radii.}
\newblock Forum of Mathematics, Sigma, Volume 2, 2014; e23.

\end{thebibliography}
\end{document}